\documentclass[11pt,a4paper]{amsart}
\usepackage{amssymb,amsmath}
\usepackage{mathrsfs}
\usepackage{newtxtext}
\usepackage[alphabetic]{amsrefs}
\usepackage{graphicx,color}
\usepackage[utf8]{inputenc}
\usepackage{amsthm}
\usepackage{fullpage}
\usepackage{latexsym}
\usepackage{slashed}
\usepackage[all]{xy}
\usepackage{hyperref}
\usepackage[mathscr]{eucal}
\usepackage{tikz}
\usetikzlibrary{calc,decorations.markings}  
\usepackage{mathtools}
\usepackage{accents}
\usepackage{amsaddr}
\usepackage{cleveref}

\def\theequation{\thesection.\arabic{equation}}
\makeatletter
\@addtoreset{equation}{section}
\makeatother

\makeatletter
\newcommand{\eqnum}{\refstepcounter{equation}\textup{\tagform@{\theequation}}}
\makeatother

\newcounter{copy}
\makeatletter
\renewcommand{\thecopy}{\ifnum0=\c@section\arabic{copy}\else\thesection.\arabic{copy}'\fi}
\makeatother

\BibSpec{book}{%
    +{}  {\PrintPrimary}                {transition}
    +{.} { \textit}                     {title}
    +{.} { }                            {part}
    +{:} { \textit}                     {subtitle}
    +{,} { \PrintEdition}               {edition}
    +{}  { \PrintEditorsB}              {editor}
    +{,} { \PrintTranslatorsC}          {translator}
    +{,} { \PrintContributions}         {contribution}
    +{,} { }                            {series}
    +{,} { \voltext}                    {volume}
    +{,} { }                            {publisher}
    +{,} { }                            {organization}
    +{,} { }                            {address}
    +{,} { }                            {status}
    +{,} { \PrintDOI}                   {doi}
    +{,} { \PrintISBNs}                 {isbn}
    +{}  { \parenthesize}               {language}
    +{}  { \PrintTranslation}           {translation}
    +{;} { \PrintReprint}               {reprint}
    +{,} { \PrintDate}                  {date}
    +{.} { }                            {note}
    +{.} {}                             {transition}
}
\BibSpec{article}{%
    +{}  {\PrintAuthors}                {author}
    +{,} { \textit}                     {title}
    +{.} { }                            {part}
    +{:} { \textit}                     {subtitle}
    +{,} { \PrintContributions}         {contribution}
    +{.} { \PrintPartials}              {partial}
    +{,} { }                            {journal}
    +{}  { \textbf}                     {volume}
    +{}  { \PrintDatePV}                {date}
    +{,} { \issuetext}                  {number}
    +{,} { \eprintpages}                {pages}
    +{,} { }                            {status}
    +{,} { \PrintDOI}                   {doi}
    +{}  { \parenthesize}               {language}
    +{}  { \PrintTranslation}           {translation}
    +{;} { \PrintReprint}               {reprint}
    +{.} { }                            {note}
    +{.} {}                             {transition}
}
\BibSpec{collection.article}{%
    +{}  {\PrintAuthors}                {author}
    +{,} { \textit}                     {title}
    +{.} { }                            {part}
    +{:} { \textit}                     {subtitle}
    +{,} { \PrintContributions}         {contribution}
    +{,} { \PrintConference}            {conference}
    +{}  {\PrintBook}                   {book}
    +{,} { }                            {booktitle}
    +{,} { \PrintDateB}                 {date}
    +{,} { pp.~}                        {pages}
    +{,} { }                            {publisher}
    +{,} { }                            {organization}
    +{,} { }                            {address}
    +{,} { }                            {status}
    +{,} { \PrintDOI}                   {doi}
    +{,} { \eprint}        {eprint}
    +{}  { \parenthesize}               {language}
    +{}  { \PrintTranslation}           {translation}
    +{;} { \PrintReprint}               {reprint}
    +{.} { }                            {note}
    +{.} {}                             {transition}
}

\BibSpec{misc}{
  +{}{\PrintAuthors}  {author}
  +{,}{ \textit}     {title}
  +{}{ (}             {date}
  +{),}{ }             {note}
  +{.}{}              {transition}
}

\usepackage{cleveref}
\theoremstyle{definition}
\newtheorem{defn}[equation]{Definition}

\newtheorem{para}[equation]{}
\theoremstyle{plain}
\newtheorem{thm}[equation]{Theorem}

\newtheorem{lem}[equation]{Lemma}
\newtheorem{cor}[equation]{Corollary}

\theoremstyle{remark}
\newtheorem{rmk}[equation]{Remark}

\crefname{defn}{Definition}{Definitions}
\crefname{notn}{Notation}{Notations}
\crefname{assmp}{Assumption}{Assumptions}
\crefname{thm}{Theorem}{Theorems}
\crefname{prp}{Proposition}{Propositions}
\crefname{lem}{Lemma}{Lemmas}
\crefname{cor}{Corollary}{Corollaries}
\crefname{conj}{Conjecture}{Conjectures}
\crefname{rmk}{Remark}{Remarks}
\crefname{exmp}{Example}{Examples}
\crefname{section}{Section}{Sections}
\crefname{subsection}{Subsection}{Subsections}
\crefname{para}{}{}
\crefname{appendix}{Appendix}{Appendices}
\crefname{subappendix}{Appendix}{Appendices}
\crefname{table}{Table}{Tables}

\newcommand{\bB}{\mathbb{B}}
\newcommand{\bC}{\mathbb{C}}

\newcommand{\bK}{\mathbb{K}}

\newcommand{\bN}{\mathbb{N}}

\newcommand{\bR}{\mathbb{R}}

\newcommand{\cV}{\mathcal{V}}

\newcommand{\vol}{\mathrm{vol}}

\newcommand{\Cl}{\text{\it C$\ell$}}

\DeclareMathOperator{\KO}{\mathrm{KO}}

\DeclareMathOperator{\Ind}{\mathrm{Ind}}
\DeclareMathOperator{\supp}{\mathrm{supp}}

\newcommand{\Sgn}{\mathrm{Sgn}}
\newcommand{\sys}{\mathrm{sys}}

\author{Yosuke Kubota}
\address{Department of Mathematical Sciences, Shinshu University\\ 3-1-1 Asahi, Matsumoto, Nagano, 390-8621, Japan\\ and \\ RIKEN iTHEMS \\ 2-1 Hirosawa, Wako, Saitama, 351-0198, Japan}
\email{ykubota@shinshu-u.ac.jp}
\title{Band width and the Rosenberg index}

\date{\today}
\begin{document}
\maketitle
\begin{abstract}
A Riemannian manifold is said to have infinite $\mathcal{KO}$-width if it admits an isometric immersion of an arbitrarily wide Riemannian band whose inward boundary has non-trivial higher index.  
In this paper we prove that if a closed spin manifold has inifinite $\mathcal{KO}$-width, then its Rosenberg index does not vanish. This gives a positive answer to a conjecture by R.~Zeidler. 
We also prove its `multi-dimensional' generalization; if a closed spin manifold admit an isometric immersion of an arbitrarily wide cube-like domain whose lowest dimensional corner has non-trivial higher index, then the Rosenberg index of $M$ does not vanish.  
\end{abstract}


\section{Introduction}\label{section:1}
The existence of a positive scalar curvature (psc) metric on a given manifold has been a fundamental problem in high-dimensional differential topology. 
An effective approach is the Dirac operator method, in which the Schr\"{o}dinger--Lichnerowicz theorem reduces the problem to the invertibility of the Dirac operator.
When the manifold is not compact, the invertibility of the Dirac operator is obstructed by the higher index, a generalization of the Fredholm index defined by using C*-algebra K-theory and coarse geometry \cites{roeLecturesCoarseGeometry2003,willett_yu_2020}. 
When one consider the universal covering of a closed manifold, the higher index of the Dirac operator is called the Rosenberg index \cites{rosenbergAstAlgebrasPositive1983,rosenbergAstAlgebrasPositive1986,rosenbergAstAlgebrasPositive1986a} and is known to be a powerful obstruction to a psc metric. 
Indeed, the Rosenberg--Stolz theorem \cites{rosenbergStableVersionGromovLawson1995,stolzManifoldsPositiveScalar2002} states that the Rosenberg index is a complete obstruction to positive scalar curvature in a stable sense, under the assumption of the Baum--Connes injectivity. 
More precisely, the vanishing of the Rosenberg index is equivalent to the existence of a psc metric after taking the direct product with sufficiently many copies of the Bott manifold (an $8$-dimensional closed spin manifold with $\Sgn(B)=0$ and $\hat{A}(B)=1$). 
On the other hand, Schick \cite{schickCounterexampleUnstableGromovLawsonRosenberg1998} constructed a closed spin manifold in dimensions $5$, $6$, $7$ which does not admit any psc metric but its Rosenberg index vanishes, by using the Schoen--Yau minimal surface method \cite{schoenStructureManifoldsPositive1979}. 
This leads us to explore a psc obstruction beyond the Rosenberg index. 
A guideline is Schick's meta-conjecture~\cite{schickTopologyPositiveScalar2014}*{Conjecture 1.5}, stating that any topological obstruction to positive scalar curvature coming from the Dirac operator is dominated by the Rosenberg index. 
For example, the psc obstructions given by the Rosenberg index of certain submanifolds of codimension 1 \cites{zeidlerIndexObstructionPositive2017,kubotaRelativeMishchenkoFomenkoHigher2020} and codimension $2$ \cites{hankeCodimensionTwoIndex2015,kubotaRelativeMishchenkoFomenkoHigher2019,kubotaGromovLawsonCodimensionObstruction2020,kubotaCodimensionTransferHigher2021} provide evidences to this conjecture.

Recently, a series of papers by Gromov \cites{gromovDozenProblemsQuestions2018,gromovMetricInequalitiesScalar2018,gromovFourLecturesScalar2019} shed new lights on this problem. 
One of the remarkable ideas proposed in these papers is the notion of a band and its width. 
A (proper) compact \emph{Riemannian band} $V$ is a compact Riemannian manifold with inward and outward boundaries $\partial _\pm V$. 
The distance of $\partial_+V$ and $\partial_-V$ is called its \emph{width}.  
Gromov proved that, if a compact Riemannian band $V$ is endowed with a psc metric but $\partial _+ V$ does not admit any psc metric due to the minimal surface method, then the width of $V$ is bounded by a constant depending on the infimum of the scalar curvature and the dimension.
Following this line, in \cites{zeidlerBandWidthEstimates2020,zeidlerWidthLargenessIndex2020,cecchiniLongNeckPrinciple2020}, Zeidler and Cecchini proved the same band width inequality when the inward boundary $\partial_+ \widetilde{V}$ of the universal covering $\widetilde{V}$ has non-trivial higher index. 
Furthermore, another approach to this inequality based on the quantitative K-theory is developed by Guo--Xie--Yu~\cite{guoQuantitativeKtheoryPositive2020}.

This band width inequality has a qualitative application to the existence of a psc. It is considered by Gromov \cite{gromovMetricInequalitiesScalar2018}*{Section 3,4} and Zeidler \cite{zeidlerBandWidthEstimates2020}*{Section 4}, the latter of which is the main subject of this paper. The following definition of the notion of $\mathcal{KO}$-width looks a little different, but equivalent, to the original definition by Zeidler. 
\begin{defn}[{\cite{zeidlerBandWidthEstimates2020}*{Definition 4.3}}]
A compact Riemannian band is said to be in the class $\mathcal{KO}$ if it is equipped with a spin structure and the index of the higher index $\Ind _{\pi_1(V)}(\slashed{D}_{\partial _+ \widetilde{V}})$ does not vanish, where $\partial _+\widetilde{V}$ denotes the inward boundary of the universal covering of $V$. For a closed manifold $M$, its $\mathcal{KO}$-width $\mathrm{width}_{\mathcal{KO}} (M,g)$ is the supremum of the width of bands in the class $\mathcal{KO}$ which is isometrically immersed to $M$. 
\end{defn}
Note that the infiniteness of the $\mathcal{KO}$-width depends only on the diffeomorphism class of $M$, i.e., is independent of the metric on it.  
The band width inequality implies that a closed manifold with infinite $\mathcal{KO}$-width does not admit a psc metric.
Since the infiniteness of the $\mathcal{KO}$-width is stable under the direct product with the Bott manifold, the Rosenberg--Stolz theorem shows the non-vanishiing of the Rosenberg index of $M$, if we assume that $\pi_1(M)$ satisfies the Baum--Connes injectivity. 
Zeidler conjectured that this non-vanishing still holds without assuming the Baum--Connes injectivity \cite{zeidlerBandWidthEstimates2020}*{Conjecture 4.12}, following the line of Schick's meta-conjecture. 
The aim of this paper is to give a positive answer to this conjecture. 

In this paper we work in a little more general setting. We need not assume that the target manifold $M$ is closed. Instead, we assume that the universal covering $\widetilde{M}$ of $M$ has a well-behaved `uniform' topology.  For $x \in X$ and $R>0$, let $B_R(x)$ denote the ball with the center $x$ and the radius $R$.  
\begin{defn}\label{defn:unif1conn}
A metric space $X$ is said to be \emph{uniformly $1$-connected} if there is an increasing function $\varphi \colon \bR_{>0} \to \bR_{>0}$ with $\varphi(t) \to \infty$ as $t \to \infty$ such that any loop in $B_R(x)$ is trivial in $B_{\varphi(R)}(x)$.  
\end{defn}
Here we extend the notion of $\mathcal{KO}$-width to complete Riemannian manifolds. Then its infiniteness depends only on the diffeomorphism class of $M$ and the coarse equivalence class of the metric. Now we state the first main theorem of this paper.
\begin{thm}\label{thm:main}
Let $(M,g)$ be a complete Riemannian spin manifold whose universal covering $\widetilde{M}$ is uniformly $1$-connected. 
If $M$ has infinite $\mathcal{KO}$-width, then the maximal equivariant coarse index $\Ind_{\Gamma} (\slashed{D}_{\widetilde{M}})$ of the Dirac operator $\slashed{D}_{\widetilde{M}}$ on $\widetilde{M}$ does not vanish.  
\end{thm}

\begin{cor}
Let $M$ be a closed spin manifold. 
If $M$ has infinite $\mathcal{KO}$-width, then the Rosenberg index $\alpha _\Gamma(M) $ does not vanish.  
\end{cor}
We remark that this corollary reproves \cite{kubotaGromovLawsonCodimensionObstruction2020}*{Theorem 1.1} as a special case, as is pointed out by Zeidler in \cite{zeidlerBandWidthEstimates2020}*{Example 4.9}. 

There are two ingredients of the proof. One is the asymptotic method in C*-algebra K-theory. 
The asymptotic C*-algebras, typically the quotient of the direct product of a sequence of C*-algebras by the direct sum, has been exploited in many researches of higher index theory such as Hanke--Schick~\cites{hankeEnlargeabilityIndexTheory2006,hankeEnlargeabilityIndexTheory2007}, Gong--Wang--Yu~\cite{gongGeometrizationStrongNovikov2008}, and so on.
In this paper, the C*-algebra of this kind leads us to a qualitative treatment of the higher index of a sequence $\{ V_n\}_{n \in \bN}$ of Riemannian bands getting wider as $n \to \infty$, which suits our purpose (although the quantitative estimate is the highlight of this new research direction after Gromov). 

Another ingredient is an estimate of the (relative) systole of Riemannian bands immersed to $M$. 
What makes the problem seem difficult is that the immersion $V \to M$ does not induce the injection of fundamental groups in general.
However, we show in \cref{lem:systole_propagation} that the length of a non-trivial loop in $V$, which is apart from the boundary and trivial in $M$, is bounded below by a constant depending on the width of $V$.
This enables us to lift an operator on $\widetilde{M}$ to the universal covering $\widetilde{V}$ of $V$ `modulo the boundary', which is a variant of the lifting lemma developed in \cite{kubotaCodimensionTransferHigher2021}.

We also discuss a generalization of \cref{thm:main}, in which we consider the multiwidth of cube-like domains immersed to $M$ instead of the width of bands. 
We say that a $\square^m$-domain is a Riemannian manifold with corner $V$ equipped with a well-behaved corner-preserving smooth map $f \colon V \to [-1,1]^m$ (more precisely, see \cref{defn:square}). We write $\partial _{j,\pm}V$ for the inverse image $f^{-1}(\{ x_j = \pm 1 \})$.
The multiwidth of $V$ is defined by the minimum of the distances of $\partial_{j,+}V$ and $\partial_{j,-}V$. 
In recent researches, a generalization of the band width inequality to such domains has been considered; 
if $V$ is a $\square^m$-domain equipped with a psc metric but its lowest dimensional corner $V_\pitchfork$ (more precisely, a transverse intersection of $m$ hypersurfaces each of which separates $\partial _{j,+}V$ with $\partial _{j,-}V$) does not admit a psc metric due to the Dirac operator or minimal surface methods, 
then its multiwidth is bounded by a constant depending on the infimum of the scalar curvature and the dimension. 
This inequality, called the $\square^{n-m}$-theorem or the $\square^{n-m}$-inequality, is proposed and proved for $\dim V \leq 8$ by Gromov \cite{gromovFourLecturesScalar2019}*{p.260}, and 
generalized to manifolds with all dimensions
by Wang--Xie--Yu and Xie \cites{wangProofGromovCube2021,xieQuantitativeRelativeIndex2021} in the case that the psc metric on $V_\pitchfork$ is obstructed by its Rosenberg index. 
A qualitative consequence of this inequality is that if a closed spin manifold $M$ has infinite $\mathcal{KO}_m$-multiwidth (\cref{defn:squareKO}), then $M$ does not admit any psc metric.
The second main theorem of this paper is to dominate this obstruction to positive scalar curvature by the Rosenberg index.
\begin{thm}\label{thm:main2}
Let $(M,g)$ be a complete Riemannian spin manifold whose universal covering $\widetilde{M}$ is uniformly $1$-connected. If $M$ has infinite $\mathcal{KO}_m$-multiwidth (\cref{defn:squareKO}), then the maximal coarse index $\Ind_\Gamma(\slashed{D}_{\widetilde{M}}) \in \KO_d(C^*(\widetilde{M})^\Gamma)$ of the Dirac operator on $\widetilde{M}$ does not vanish. 
\end{thm}
\begin{cor}
Let $(M,g)$ be a closed Riemannian spin manifold. If $M$ has infinite $\mathcal{KO}_m$-width, then the Rosenberg index $\alpha_\Gamma(M)$ does not vanish. 
\end{cor}

\subsection*{Acknowledgement}
The author would like to thank Shinichiroh Matsuo for a helpful comment.
This work was supported by RIKEN iTHEMS and JSPS KAKENHI Grant Numbers 19K14544, JPMJCR19T2, 17H06461.

\section{Proof of \cref{thm:main}}\label{section:main}
\subsection{Systole of Riemannian bands immersed to a uniformly $1$-connected manifold}
The first step for the proof is an observation on the systole, the minimum of the length of homotopically non-trivial loop, of a subspace of a uniformly $1$-connected manifold.

Let $M$ be a $d$-dimensional complete Riemannian spin manifold with infinite $\mathcal{KO}$-width. 
For each $n \in \bN$, pick a compact Riemannian band $V_n$ in the class $\mathcal{KO}$ whose width is greater than $n$, and an isometric immersion $f_n \colon V_n \to M$. 
Let $\pi_n:= \pi_1(V_n)$ and let $\Gamma _n \subset \Gamma$ denote the image of $\pi_n$ under the inclusion $V_n \to M$.
Then the map $f_n$ lifts to a codimension $0$ immersion $\tilde{f}_n \colon \widetilde{V}_n \to \widetilde{M}$, where $\widetilde{V}_n$ is the universal covering of $V_n$. 
In short, we write $\mathbf{V}$ and $\widetilde{\mathbf{V}}$ for the disjoint union $\bigsqcup_{n \in \bN} V_n$ and $\bigsqcup_{n \in \bN} \widetilde{V}_n $ respectively. 
For $R>0$, let $\widetilde{V}_{n,R}$ denote the open subset of $\widetilde{V}_n$ consisting of points $x \in \widetilde{V}_n$ such that $d(x, \partial \widetilde{V}) \geq R$. 

Let $\widetilde{U}_n$ and $\widetilde{U}_{n,R}$ denote the interior of $\tilde{f}(\widetilde{V}_n)$ and $\tilde{f}(\widetilde{V}_{n,R})$ respectively, which are non-empty open subsets of $\widetilde{M}$. 
For an inclusion $Y \subset X$ of length spaces, we call the infimum of the length of closed loops in $Y$ representing non-trivial homotopy class in $X$ the \emph{relative systole} of $Y \subset X$ and write $\sys (Y \subset X)$. This corresponds to the systole of $Y$ relative to $\pi_1(Y) \to \pi_1(X)$ in the standard terminology of systolic geometry (cf.~\cite{katzSystolicGeometryTopology2007}*{Definition 8.2.1}).  
\begin{lem}\label{lem:systole_propagation}
Assume that $\widetilde{M}$ is uniformly $1$-connected with respect to a function $\varphi \colon \bR_{>0} \to \bR_{>0}$ and let $C_R :=\inf \{ C>0\mid \varphi(C) \geq R\} $. 
Then we have  $\sys (\widetilde{U}_{n,R} \subset \widetilde{U}_n) \geq 2C_R$, where the relative systole is defined with respect to the Riemannian distance of $\widetilde{M}$. 
\end{lem}
\begin{proof}
Let $\ell \colon S^1 \to \widetilde{U}_{n,R}$ be a closed loop with $\mathrm{length}(\ell) \leq 2C_R$. Let $x:=\ell(0)$. 
Note that $\ell$ is contained in the open ball $B_{C_R}(x)$.
By the assumption of uniform $1$-connectedness of $\widetilde{M}$, this $\ell$ is null-homotopic in $B_{\varphi(C_R)}(x) \subset B_R(x)$. Since $x \in \widetilde{U}_{n,R}$, we have $ B_{R}(x) \subset \widetilde{U}_{n}$. Hence we obtain that $[\ell] \in \pi_1(\widetilde{U}_{n})$ is trivial. This shows that $\mathrm{sys}(\widetilde{U}_{n,R} \subset \widetilde{U}_{n}) \geq 2 C_R$.  
\end{proof}

The following lemma is standard in coarse geometry, rather known as the uniform contractibility of the universal covering of an aspherical manifold (see e.g.~\cite{roeLecturesCoarseGeometry2003}*{Example 5.26}).
\begin{lem}\label{lem:universal_covering}
Let $(M,g)$ be a closed Riemannian manifold. Then the universal covering $\widetilde{M}$ is uniformly $1$-connected. 
\end{lem}
\begin{proof}
For $R>0$ and $x \in \widetilde{M}$, let $C_{x,R}$ denote the infimum of the real numbers $C \geq R$ such that the inclusion $B_R(x) \to B_C(x)$ induces the trivial map in $\pi_1$-groups. 
We show that $C_{x,R} < \infty$ for any $x$ and $R$. 
Let $r>0 $ be less than the injectivity radius of $M$. 
Then $B_R(x)$ is covered by a finite number of open balls $B_r(y_i)$ for $i=1,\cdots, k$. 
The open subspace $U:=\bigcup_{i=1}^k B_r(y_i) \subset \widetilde{M}$ is homotopy equivalent to its nerve, and hence has the homotopy type of a finite simplicial complex. 
In particular, its fundamental group is finitely generated, and hence there is $L>0$ such that $U \subset B_L(x)$ and the induced map $\pi_1(U) \to \pi_1(B_L(x))$ is trivial. 
This shows $C_{x,R}\leq L <\infty$, as desired. 

The assignment $x \mapsto C_{x,R}$ is $\Gamma$-invariant and locally bounded (indeed, $C_{y,R} \leq C_{x,R+\varepsilon} $ for any $y \in B_\varepsilon(x)$).
Therefore, together with compactness of $M=\widetilde{M}/\Gamma$, we get $\varphi (R):= \sup_{x \in \widetilde{M}} (C_{x,R} +1) < \infty$. 
This $\varphi$ is the desired function since it is by definition increasing.
\end{proof}

\subsection{Lifting finite propagation operators}
Next, we construct a lift of an operator on $\widetilde{M}$ to the non-compact Riemannian bands $\widetilde{V}_n$, which forms a $\ast$-homomorphism `modulo boundary'. This technique is inherited from \cite{kubotaCodimensionTransferHigher2021}.

Let $T \in \bB(L^2(\widetilde{M}))$ which is locally Hilbert--Schmidt, i.e., $Tf$, $fT$ is in the Hilbert--Schmidt class for any $f \in C_c(\widetilde{M})$.  
Such $T$ is represented by a kernel function $T \colon \widetilde{M} \times \widetilde{M} \to \bC$ as
\[T \xi (x):= \int_{y \in \widetilde{M}} T(x,y)\xi(y)d \vol _g(y). \]
The support of $T$ is defined as the support of its kernel function in $ \widetilde{M} \times \widetilde{M}$, and $\mathrm{Prop} (T):=\sup\{  d(x,y) \mid (x,y) \in \supp T \} $ is called the propagation of $T$, where $d$ denotes the metric on $\widetilde{M}$. 
We define $\bC_{\mathrm{HS}} [\widetilde{M}]^\Gamma$ as the Real $\ast$-algebra of $\Gamma$-invariant locally Hilbert--Schmidt operators with finite propagation. Let $C^*(\widetilde{M})^\Gamma$ denote the completion of $\bC_{\mathrm{HS}}[\widetilde{M}]^\Gamma$ with respect to the maximal norm satisfying the C*-condition $\| T^*T\| =\|T\|^2$ (which is well-defined by \cite{gongGeometrizationStrongNovikov2008}*{3.5}). 
In the same way, we also define $\bC_{\mathrm{HS}}[\widetilde{V}_n]^{\pi_n} \subset \bB(L^2(\widetilde{V}_n))$ and its completion $C^*(\widetilde{V}_n)^{\pi_n}$.

We define the Real $\ast$-algebra $\bC_{\mathrm{HS}}[\widetilde{\mathbf{V}}]^{\boldsymbol{\pi}}$ consisting of sequences of locally compact operators $T_n \in \bB(L^2(\widetilde{V}_n))$ with a uniform bound of propagation, i.e., 
\[ \bC_{\mathrm{HS}}[\widetilde{\mathbf{V}}]^{\boldsymbol{\pi}} := \Big\{ (T_n) \in \prod_{n \in \bN} \bC_{\mathrm{HS}}[\widetilde{V}_n]^{\pi_n}  \ \Big| \ \text{ $\exists R>0$ such that $\mathrm{Prop}(T_n) <R$ for any $n$}  \Big\} \]
and let $C^*(\widetilde{\mathbf{V}})^{\boldsymbol{\pi}}$ denote its closure in the direct product C*-algebra $\prod _{n \in \bN} C^*(\widetilde{V}_n)^{\pi_n}$. We also define the Real C*-ideal $C^*(\partial \widetilde{\mathbf{V}} \subset \widetilde{\mathbf{V}})^{\boldsymbol{\pi}} \triangleleft C^*(\widetilde{\mathbf{V}})^{\boldsymbol{\pi}}$ as the closure of 
\[ \bC_{\mathrm{HS}}[\partial \widetilde{\mathbf{V}}\subset \widetilde{\mathbf{V}}]^{\boldsymbol{\pi}} := \bigg\{ (T_n) \in \prod_{n \in \bN} \bC_{\mathrm{HS}}[\widetilde{V}_n]^{\pi_n}  \ \bigg| \ \text{ \parbox{36ex}{$\exists R>0$ such that $\mathrm{Prop}(T_n) <R$ and $d(\supp T_n, \partial \widetilde{V}_n \times \partial \widetilde{V}_n) <R$ for any $n$} } \bigg\}_{\textstyle ,} \]
where $d(\supp T_n, \partial \widetilde{V}_n \times \partial \widetilde{V}_n) $ stands for the Hausdorff distance.

Let $D^*(\widetilde{M})^\Gamma$ denote the closure of the set of bounded operators on $L^2(\widetilde{M})$ which is of finite propagation and is pseudo-local, i.e., $[T,f] \in \bK$ for any $f \in C_c(\widetilde{M})$, with respect to the norm 
\[ \| T\|_{D^*(\widetilde{M})^\Gamma}:=\sup _{S \in \bC_{\mathrm{HS}}[\widetilde{M}]^\Gamma \setminus \{0\}} \frac{ \| TS\| _{C^*(\widetilde{M})^\Gamma} }{ \| S\|_{C^*(\widetilde{M})^\Gamma}}.\] 
This is a Real C*-algebra \cite{oyono-oyonoTheoryMaximalRoe2009}*{Lemma 2.16}, including $C^*(\widetilde{M})^\Gamma$ as a Real C*-ideal. We write the quotient as $Q^*(\widetilde{M})^\Gamma := D^*(\widetilde{M})^\Gamma / C^*(\widetilde{M})^\Gamma$. 
A standard fact in coarse index theory is that the K-group of $Q^*(\widetilde{M})^\Gamma$ is isomorphic to the equivariant K-homology $\KO^\Gamma_*(\widetilde{M})$. 
In the same way, we also define the Real C*-algebras $D^*(\widetilde{\mathbf{V}})^{\boldsymbol{\pi}}$, $Q^*(\widetilde{\mathbf{V}})^{\boldsymbol{\pi}}$,  $D^*(\partial \widetilde{\mathbf{V}} \subset \widetilde{\mathbf{V}})^{\boldsymbol{\pi}}$, and $Q^*(\partial \widetilde{\mathbf{V}} \subset \widetilde{\mathbf{V}})^{\boldsymbol{\pi}}$. We only note that $D^*(\partial \widetilde{\mathbf{V}} \subset \widetilde{\mathbf{V}})^{\boldsymbol{\pi}}$ is the closure of the set of pseudo-local, finite propagation operators which is supported near $\partial V$ and $Tf, fT \in \bK$ for any $f \in C_c(\mathbf{V}\setminus \partial \mathbf{V})$.

\begin{para}\label{lem:lift}
There is a $\ast$-homomorphism 
\[
s \colon C^*(\widetilde{M})^\Gamma  \to C^*(\widetilde{\mathbf{V}})^{\boldsymbol{\pi}}/C^*(\partial \widetilde{\mathbf{V}} \subset \widetilde{\mathbf{V}})^{\boldsymbol{\pi}}
\]
constructed in the following way. 
For $T \in \bC_{\mathrm{HS}}[\widetilde{M}]^{\Gamma }$ with $\mathrm{Prop}(T) \leq C_R$, we define its lift $\widetilde{T}_{n,R} \in \bC_{\mathrm{HS}}[\widetilde{V}_n]^{\pi_n}$ in terms of its kernel function as
\begin{align} 
\begin{split}
\widetilde{T}_{n,R}(\tilde{x},\tilde{y}) := \begin{cases}
T(\tilde{f}_n(\tilde{x}),\tilde{f}_n(\tilde{y})) & \text{ if $d(\tilde{x}, \tilde{y}) <C_R$ and $\tilde{x}, \tilde{y} \in \widetilde{V}_{n,R}$, } \\
0 & \text{ otherwise. }
\end{cases} 
\end{split}\label{eq:lift}
\end{align}
A consequence of \cref{lem:systole_propagation} is that, for each $\tilde{x} \in \widetilde{V}_{n,R}$ and $y \in \widetilde{M}$ with $d(\tilde{f}_n(\tilde{x}), y) < C_R$, there is a unique $\tilde{y} \in \tilde{f}_n^{-1}(y)$ such that $\widetilde{T}_{n,R}(\tilde{x},\tilde{y}) \neq 0$. 
In particular, we have $\mathrm{Prop}(\widetilde{T}_{n,R}) =\mathrm{Prop}(T)$. 
Moreover, the $\Gamma_n$-invariance of $T$ implies the $\pi_n$-invariance of $\widetilde{T}_{n,R}$.

The assignment $T \mapsto \widetilde{T}_{n,R}$ is linear and $\ast$-preserving. Moreover, 
for $T,S \in \bC_{\mathrm{HS}}[\widetilde{M}]^\Gamma$ with $\mathrm{Prop}(T) + \mathrm{Prop}(S) \leq C_R$, the lifts $\widetilde{T}_{n,R}$, $\widetilde{S}_{n,R}$ and $\widetilde{TS}_{n,R}$ satisfies that 
\[ \int_{\tilde{z} \in \widetilde{V}_{n,R}}\widetilde{T}_{n,R}(\tilde{x},\tilde{z}) \widetilde{T}_{n,R}(\tilde{z},\tilde{y}) d\vol_g (z)= (\widetilde{TS})_{n,R} (\tilde{x}, \tilde{y}) \]
if $B_{C_R}(\tilde{x}) \cap B_{C_R}(\tilde{y}) \subset \widetilde{V}_{n,R}$. 
This implies that $\widetilde{T}_{n,R} \widetilde{S}_{n,R} - \widetilde{TS}_{n,R}$ is supported on the $(R + C_R)$-neighborhood of $\partial \widetilde{V}_n \times \partial \widetilde{V}_n$, i.e., the lifting is multiplicative modulo boundary.  

Recall that the assignment $R \mapsto C_R$ is increasing and $C_R \to \infty$ as $R \to \infty$. Also, for $R' > R > 0$ and $T \in \bC_{\mathrm{HS}}[\widetilde{M}]^\Gamma$ with $\mathrm{Prop}(T) \leq C_R$, the difference $\widetilde{T}_{n,R'} - \widetilde{T}_{n,R}$ is supported on the $R'$-neighborhood of $\partial \widetilde{V}_n \times \partial \widetilde{V}_n$. This shows that
\[s(T) := (\widetilde{T}_{n,R})_{n \in \bN} \in C^*(\widetilde{\mathbf{V}})^{\boldsymbol{\pi}}/C^*(\partial \widetilde{\mathbf{V}} \subset \widetilde{\mathbf{V}})^{\boldsymbol{\pi}}  \]
is well-defined independent of the choice of $R>0$ with $\mathrm{Prop}(T) < C_R$.
By the above argument, this $s$ forms a $\ast$-homomorphism from $\bC_{\mathrm{HS}}[\widetilde{M}]^\Gamma$. This extends to the $\ast$-homomorphism from $C^*(\widetilde{M})^\Gamma $ by the maximality of the norm on $C^*(\widetilde{M})^\Gamma$. 
\end{para}

\begin{para}\label{lem:liftD}
The $\ast$-homomorphism $s$ defined in \cref{lem:lift} extends to 
\[s \colon D^*(\widetilde{M})^\Gamma \to 
D^*(\widetilde{\mathbf{V}})^{\boldsymbol{\pi}}/D^*(\partial \widetilde{\mathbf{V}} \subset \widetilde{\mathbf{V}})^{\boldsymbol{\pi}}, \]
and hence induces 
\[ s \colon Q^*(\widetilde{M})^\Gamma  \to Q^*(\widetilde{\mathbf{V}})^{\boldsymbol{\pi}} / Q^*(\partial \widetilde{\mathbf{V}} \subset \widetilde{\mathbf{V}})^{\boldsymbol{\pi}}.\]
Indeed, without loss of generality, we may assume that there are Riemannian bands $V_n'$ such that $V_n \subset V_n'$, $f_n$ extends to an isometric immersion of $V_n'$, and $\mathrm{dist}(\partial V_n, \partial V_n') \geq 1$.
Let us decompose an operator $T \in D^*(\widetilde{M})^\Gamma$ into $T=T^0 + T^1$, where $\mathrm{Prop}(T^0) < C_1$ and $T^1 \in C^*(\widetilde{M})^\Gamma$. Set 
\[ s(T) := (\Pi_n (\widetilde{T^0})_{n,1}' \Pi_n )_{n \in \bN} + s(T^1) \in D^*(\widetilde{\mathbf{V}})^{\boldsymbol{\pi}}/D^*(\partial \widetilde{\mathbf{V}} \subset \widetilde{\mathbf{V}})^{\boldsymbol{\pi}},  \]
where $(\cdot)_{n,1}'$ denotes the lift \eqref{eq:lift} with respect to $V_n'$ and $R=1$, and $\Pi_n $ denotes the projection onto $L^2(V_n)$.
This $s$ is well-defined independent of the choice of a decomposition $T=T_0+T_1$, and forms a $\ast$-homomorphism. 
We omit the detail of the proof, because it is proved completely in the same way as \cite{kubotaCodimensionTransferHigher2021}*{Proposition 4.3}. 
\end{para}

\subsection{K-theory and the coarse index of Dirac operators}
In the last step, we relate the equivariant coarse index of $\widetilde{M}$ with that of  $\partial \widetilde{V}_n$ through the lifting homomorphism $s$ constructed above. 

We define the ideal $C^*_0 (\partial \widetilde{\mathbf{V}} \subset \widetilde{\mathbf{V}})^{\boldsymbol{\pi}}$ of $C^*(\partial \widetilde{\mathbf{V}} \subset \widetilde{\mathbf{V}})^{\boldsymbol{\pi}}$ consisting of operators $(T_n)_n$ such that $\| T_n\| \to 0$ as $n \to \infty$, and set
\[C^*_\flat (\partial \widetilde{\mathbf{V}} \subset \widetilde{\mathbf{V}})^{\boldsymbol{\pi}} := C^* (\partial \widetilde{\mathbf{V}} \subset \widetilde{\mathbf{V}})^{\boldsymbol{\pi}} /C^*_0 (\partial \widetilde{\mathbf{V}} \subset \widetilde{\mathbf{V}})^{\boldsymbol{\pi}}. \]
We also define the Real C*-algebras $D^*_\flat (\partial \widetilde{\mathbf{V}} \subset \widetilde{\mathbf{V}})^{\boldsymbol{\pi}}$ and $Q^*_\flat (\partial \widetilde{\mathbf{V}} \subset \widetilde{\mathbf{V}})^{\boldsymbol{\pi}}$ in the same way. 
\begin{lem}\label{lem:K}
There are isomorphisms
\begin{align*}
   \phi  \colon \KO_* ( C^*_\flat (\partial \widetilde{\mathbf{V}} \subset \widetilde{\mathbf{V}})^{\boldsymbol{\pi}} ) &\cong \frac{\prod \KO_*(C^*(\partial_+ \widetilde{V}_n)^{\pi_n})}{\bigoplus \KO_*(C^*(\partial_- \widetilde{V}_n)^{\pi_n})} \oplus \frac{\prod \KO_*(C^*(\partial_- \widetilde{V}_n)^{\pi_n})}{\bigoplus \KO_*(C^*(\partial_- \widetilde{V}_n)^{\pi_n})},\\ 
    \phi \colon \KO_* ( D^*_\flat (\partial \widetilde{\mathbf{V}} \subset \widetilde{\mathbf{V}})^{\boldsymbol{\pi}} ) &\cong \frac{\prod \KO_*(D^*(\partial_+ \widetilde{V}_n)^{\pi_n})}{\bigoplus \KO_*(D^*(\partial_+ \widetilde{V}_n)^{\pi_n})} \oplus \frac{\prod \KO_*(D^*(\partial_- \widetilde{V}_n)^{\pi_n})}{\bigoplus \KO_*(D^*(\partial_- \widetilde{V}_n)^{\pi_n})},\\
    \phi \colon\KO_* ( Q^*_\flat (\partial \widetilde{\mathbf{V}} \subset \widetilde{\mathbf{V}})^{\boldsymbol{\pi}} ) &\cong \frac{\prod \KO_*(Q^*(\partial_+ \widetilde{V}_n)^{\pi_n})}{\bigoplus \KO_*(Q^*(\partial_+ \widetilde{V}_n)^{\pi_n})} \oplus \frac{\prod \KO_*(Q^*(\partial_- \widetilde{V}_n)^{\pi_n})}{\bigoplus \KO_*(Q^*(\partial_- \widetilde{V}_n)^{\pi_n})}.
\end{align*}  
\end{lem}
We write $\phi_\pm $ for the first and the second component of $\phi$ respectively. 
\begin{proof}
Let $V_{n,\pm } \subset V_{n}$ denote the $n/4$-neighborhood of $\partial _\pm V$ and let $\mathbf{V}_\pm := \bigsqcup _{n \in \bN} V_{n,\pm}$. 
Then the inclusion
\[ C^*(\partial_+ \widetilde{\textbf{V}} \subset \widetilde{\textbf{V}}_{+})^{\boldsymbol{\pi}} \oplus C^*(\partial_- \widetilde{\textbf{V}} \subset \widetilde{\textbf{V}}_{-}) ^{\boldsymbol{\pi}} \to C^*(\partial \widetilde{\textbf{V}} \subset \widetilde{\textbf{V}})^{\boldsymbol{\pi}} \]
induces a $\ast$-isomorphism
\[ C^*_\flat (\partial_+ \widetilde{\textbf{V}} \subset \widetilde{\textbf{V}}_{+})^{\boldsymbol{\pi}}\oplus C^*_\flat (\partial_- \widetilde{\textbf{V}} \subset \widetilde{\textbf{V}}_{-})^{\boldsymbol{\pi}} \to C^*_\flat (\partial \widetilde{\textbf{V}} \subset \widetilde{\textbf{V}})^{\boldsymbol{\pi}}. \]
Now the six-term exact sequence for the extension 
\[ 0 \to C^*_0(\partial_\pm  \widetilde{\mathbf{V}} \subset \widetilde{\mathbf{V}}_\pm )^{\boldsymbol{\pi}}  \to C^*(\partial_\pm  \widetilde{\mathbf{V}} \subset \widetilde{\mathbf{V}}_\pm)^{\boldsymbol{\pi}}  \to C^*_\flat (\partial_\pm \widetilde{\mathbf{V}} \subset \widetilde{\mathbf{V}}_\pm)^{\boldsymbol{\pi}}  \to 0\]
proves the first isomorphism. The second and the third isomorphisms are also proved in the same way. 
\end{proof}
Let 
\begin{align}
    \partial \colon \KO_*\Big( \frac{C^*(\widetilde{\mathbf{V}})^{\boldsymbol{\pi}}}{C^*(\partial \widetilde{\mathbf{V}} \subset \widetilde{\mathbf{V}})^{\boldsymbol{\pi}}}\Big)  \to \KO_{*-1}(C^*_\flat (\partial \widetilde{\mathbf{V}} \subset \widetilde{\mathbf{V}})^{\boldsymbol{\pi}}) \label{eq:boundary_flat}
\end{align}
denote the K-theory boundary map associated to the exact sequence
\[0 \to \frac{C^* (\partial \widetilde{\mathbf{V}} \subset \widetilde{\mathbf{V}})^{\boldsymbol{\pi}}}{C^*_0 (\partial \widetilde{\mathbf{V}} \subset \widetilde{\mathbf{V}})^{\boldsymbol{\pi}}}  \to \frac{C^*(\widetilde{\mathbf{V}})^{\boldsymbol{\pi}}}{C^*_0(\partial \widetilde{\mathbf{V}} \subset \widetilde{\mathbf{V}})^{\boldsymbol{\pi}}} \to \frac{C^*(\widetilde{\mathbf{V}})^{\boldsymbol{\pi}}}{C^*(\partial \widetilde{\mathbf{V}} \subset \widetilde{\mathbf{V}})^{\boldsymbol{\pi}}} \to 0. \]
We also use the same letter $\partial$ for the $\KO$-theory boundary map of the same kind defined for $D^*$ and $Q^*$ coarse C*-algebras. 

Let $\slashed{D}_{\widetilde{M}}$, $\slashed{D}_{\widetilde{V}_n}$, $\slashed{D}_{\partial_\pm  \widetilde{V}_n}$ denote the $\Cl_{d,0}$-linear Dirac operator on $\widetilde{M}$ and $\widetilde{V}_n$, and the $\Cl_{d-1,0}$-linear Dirac operator on $\partial_\pm \widetilde{V}_n$ respectively (for the definition of $\Cl_{d,0}$-linear Dirac operator on a $d$-dimensional spin manifold, see e.g.~\cite{lawsonjr.SpinGeometry1989}*{Chapter II, \S 7}). 
We also consider the Clifford-linear Dirac operators $\slashed{D}_{\widetilde{\mathbf{V}}}$ and $\slashed{D}_{\partial_\pm \widetilde{\mathbf{V}}}$, each of which is the same thing as the family $(\slashed{D}_{\widetilde{V}_n})_{n \in \bN}$ and $(\slashed{D}_{\partial _{\pm }\widetilde{V}_n})_{n \in \bN }$ respectively. 
Let $\chi \colon \bR \to \bR$ is a continuous function such that $\chi(t) \to \pm 1$ as $t \to \pm \infty$ and the support of $\hat{\chi}$ is in $(-1,1)$.
Then $\chi(\slashed{D}_{\widetilde{M}})$ is an odd self-adjoint operator in $D^*(\widetilde{M})^\Gamma$ by \cite{roePartitioningNoncompactManifolds1988}*{Proposition 2.3}, whose image in $Q^*(\widetilde{M})^\Gamma$ is unitary. Hence it determines an element of $\KO_{d+1}(Q^*(\widetilde{M})^\Gamma)$ (cf.~\cite{kubotaCodimensionTransferHigher2021}*{Remark A.1}), which is denoted by $[\slashed{D}_{\widetilde{M}}]$ in short. 
Similarly, we define $[\slashed{D}_{\partial_\pm  \widetilde{V}_n}] \in \KO_{d}(Q^*(\partial _\pm \widetilde{V}_n)^{\pi_n})$, and $[\slashed{D}_{\partial \widetilde{\mathbf{V}}}] \in \KO_{d}(Q^*(\partial \widetilde{\mathbf{V}})^{\boldsymbol{\pi}})$. 
We also define the relative KO-class of the Dirac operators on manifolds with boundary 
$[\slashed{D}_{\widetilde{V}_n}] \in \KO_{d+1}(Q^*(\widetilde{V}_n)^{\pi_n} / Q^*(\partial \widetilde{V}_n \subset \widetilde{V}_n)^{\pi_n})$ and $[\slashed{D}_{\widetilde{\mathbf{V}}}] \in \KO_{d+1}(Q^*(\widetilde{\mathbf{V}})^{\boldsymbol{\pi}}/ Q^*(\partial \widetilde{\mathbf{V}} \subset \widetilde{\mathbf{V}})^{\boldsymbol{\pi}})$.


The K-theory boundary map of the extensions $0 \to C^*(X)^G \to D^*(X)^G \to Q^*(X)^G \to 0$ are denoted by $\Ind _G$, where $(X,G)$ is $(\widetilde{M},\Gamma)$, $(V_n, \pi_n)$, $(\widetilde{\mathbf{V}}, \boldsymbol{\pi})$, and so on. 
Note that $C^*(M)^\Gamma$ is Morita equivalent to the maximal group C*-algebra $C^*\Gamma$ and the Rosenberg index $\alpha_\Gamma(M)$ is the same thing with $\Ind_\Gamma([\slashed{D}_{\widetilde{M}}]) $.

Hereafter, for a sequence of abelian groups $A_n$ and an element $(a_n)_n \in \prod A_n$, we write $(a_n)_n^\flat$ for its image by the quotient map $\prod A_n \to \prod A_n / \bigoplus A_n$.
\begin{lem}
The composition
\[\KO_{d+1}(Q^*(\widetilde{M})^\Gamma ) \xrightarrow{s_*} \KO_{d+1}\Big( \frac{Q^*(\widetilde{\mathbf{V}})^{\boldsymbol{\pi}}}{Q^*(\partial \widetilde{\mathbf{V}} \subset \widetilde{\mathbf{V}})^{\boldsymbol{\pi}}} \Big) \xrightarrow{\partial} \KO_{d} ( Q^*_\flat (\partial \widetilde{\mathbf{V}} \subset \widetilde{\mathbf{V}})^{\boldsymbol{\pi}} ) \xrightarrow{\phi_+} \frac{\prod \KO_d(Q^*(\partial_+ \widetilde{V}_n)^{\pi_n})}{\bigoplus \KO_d(Q^*(\partial_+ \widetilde{V}_n)^{\pi_n})} \]
sends $[\slashed{D}_{\widetilde{M}}]$ to $([\slashed{D}_{\partial_+ \widetilde{V}_n } ])_n^\flat$.
\end{lem}
\begin{proof}
Firstly, we have $s_*[\slashed{D}_{\widetilde{M}}] = [\slashed{D}_{\widetilde{\mathbf{V}}}]$. 
This is because the operators $s(\chi(\slashed{D}_{\widetilde{M}}))$ and $\chi (\slashed{D}_{\mathbf{V}})$ are both
$0$-th order pseudo-differential operators and their principal symbols are the same. 
Next, $\partial [\slashed{D}_{\widetilde{\mathbf{V}}}] = [\slashed{D}_{\partial _+ \widetilde{\mathbf{V}}}]$ follows from the `boundary of Dirac is Dirac' principle (see e.g.~\cite{higsonAnalyticHomology2000}*{Proposition 11.2.15}). 
Finally, $\phi_+([\slashed{D}_{\partial _+ \widetilde{\mathbf{V}}}]) = ([\slashed{D}_{\partial_+ \widetilde{V}_n } ])_n^\flat$ is obvious from the definition. 
\end{proof}

\begin{proof}[{Proof of \cref{thm:main}}]
By definition, the diagram 
\[
\xymatrix{
\KO_{d+1}(Q^*(\widetilde{M})^\Gamma ) \ar[r]^{s_*} \ar[d]^{\Ind_\Gamma } & 
\KO_{d+1}\Big( \frac{Q^*(\widetilde{\mathbf{V}})^{\boldsymbol{\pi}}}{Q^*(\partial \widetilde{\mathbf{V}} \subset \widetilde{\mathbf{V}})^{\boldsymbol{\pi}}} \Big) \ar[r]^{\partial } \ar[d]^{\Ind_{\boldsymbol{\pi}}} & 
\KO_{d} ( Q^*_\flat (\partial \widetilde{\mathbf{V}} \subset \widetilde{\mathbf{V}})^{\boldsymbol{\pi}} ) \ar[d]^{\Ind_{\boldsymbol{\pi}}}  \ar[r]^{\phi_+} & \frac{\prod \KO_d(Q^*(\partial_+ \widetilde{V}_n)^{\pi_n})}{\bigoplus \KO_d(Q^*(\partial_+ \widetilde{V}_n)^{\pi_n})} \ar[d]^{(\Ind_{\pi_n})_n^\flat} \\
\KO_d(C^*(\widetilde{M})^\Gamma ) \ar[r]^{s_*} & 
\KO_d\Big( \frac{C^*(\widetilde{\mathbf{V}})^{\boldsymbol{\pi}}}{C^*(\partial \widetilde{\mathbf{V}} \subset \widetilde{\mathbf{V}})^{\boldsymbol{\pi}}} \Big) \ar[r]^{\partial }  & 
\KO_{d-1} ( C^*_\flat (\partial \widetilde{\mathbf{V}} \subset \widetilde{\mathbf{V}})^{\boldsymbol{\pi}} ) \ar[r] ^{\phi_+}
& \frac{\prod \KO_{d-1}(C^*(\partial_+ \widetilde{V}_n)^{\pi_n})}{\bigoplus \KO_{d-1}(C^*(\partial_+ \widetilde{V}_n)^{\pi_n})}
}
\]
commutes. Therefore, we get
\begin{align*}
    (\phi_+ \circ \partial \circ s_*) (\Ind _\Gamma ([\slashed{D}_{\widetilde{M}}])) &= (\Ind_{\pi_n})_n^\flat ((\phi_+ \circ \partial \circ s_*)([\slashed{D}_{\widetilde{M}}])) =  (\Ind _{\pi_n} ([\slashed{D}_{\partial_+ \widetilde{V}_n} ]))_n ^\flat . 
\end{align*}  
The right hand side is non-zero by assumption. This shows the non-vanishing of $\Ind _\Gamma ([\slashed{D}_{\widetilde{M}}])$ as desired. 
\end{proof}

\section{Multiwidth of cube-like domains and the Rosenberg index}\label{section:cube}
In this section we study a `multi-dimensional' generalization of \cref{thm:main}.
\begin{defn}\label{defn:square}
A $d$-dimensional Riemannian $\square^{m}$-domain is a compact Riemannian manifold $(V,g)$ with corners equipped with a face-preserving smooth map $f \colon V \to [-1,1]^m$ which is corner proper, i.e., $1$-faces of $V$ are pull-backs of $1$-faces of $[-1,1]^m$ (cf.~\cite{gromovFourLecturesScalar2019}*{3.18}).  
\end{defn}
In this terminology, $\square^1$-domain is the same thing as Riemannian band. 

We write the codimension $1$ faces of $V$ as $\partial_{j,\pm} V := f^{-1}(p_j^{-1}(\{ \pm 1 \} ))$. 
Note that the codimension $m$ corner 
\[ V_\pitchfork := \partial_{1,+}V \cap \cdots \cap \partial _{m,+}V \]
is a closed manifold.
\begin{defn}
We define the multiwidth of a $\square^{m}$-domain as
\[ \mathrm{width} (V,g):= \min_{j=1,\dots ,m} \mathrm{dist}(\partial _{j,+} V, \partial _{j,-}V). \]
\end{defn}
For a class $\cV_m$ of Riemannian $\square^{m}$-domains, we define the $\cV_m$-multiwidth of a complete spin manifold $M$, denoted by $\mathrm{width}_{\cV_m} (M, g)$, as the supremum of the width of a Riemannian $\square^{n-m}$-domains immersed to $M$.  

\begin{defn}\label{defn:squareKO}
We say that a $\square^{m}$-domain $V$ is in the class $\mathcal{KO}_m$ if the equivariant coarse index 
\[ \Ind_{\pi_1(V)}(\slashed{D}_{\widetilde{V}_\pitchfork}) \in \KO_{d-m}(C^*(\widetilde{V}_{\pitchfork})^{\pi_1(V)})\] 
does not vanish. 
\end{defn}

\begin{rmk}
We compare our assumption on cube-like domains with the previous works \cites{gromovFourLecturesScalar2019,wangProofGromovCube2021,xieQuantitativeRelativeIndex2021}. 
First, the above papers deal with a manifold with boundary $X$, instead of a manifold with corner, equipped with a map $f \colon X \to [-1,1]^m$ sending $\partial X$ to the boundary of the cube. 
For such $X$, the inverse image $X':=f^{-1}([-1+\varepsilon, 1-\varepsilon]^m)$ is a $\square^m$-domain in the sense of \cref{defn:square} if $\varepsilon >0$ is chosen to be a regular value of $p_j \circ f$ for $j=1,\dots, m$ (where $p_j$ denotes the $j$-th projection). 
We may choose such $\varepsilon >0$ in the way that the distance of $\partial X$ and $\partial X'$ is arbitrarily small.

Next, the submanifold playing the role of $\partial_+ V$ in the band width theory in the previous papers is not the lowest dimensional corner $V_\pitchfork$, but the transverse intersection $Y_\pitchfork$ of $m$ hypersurfaces $Y_j \subset V$ which separates $\partial_{j,+}V$ and $\partial _{j,-}V$.
In particular, the assumption of \cites{wangProofGromovCube2021,xieQuantitativeRelativeIndex2021} is the non-vanishing of the Rosenberg index of $Y_\pitchfork$ under the condition that $\pi_1(Y_\pitchfork) \to \pi_1(V)$ is injective. 
Indeed, this assumption is reduced to the non-vanishing of the $\pi_1(V)$-equivariant coarse index of the $\pi_1(V)$-Galois covering $\widetilde{Y}_\pitchfork$ (instead of $\pi_1(Y_\pitchfork)$-equivariant coarse index of the universal covering of $Y_\pitchfork$). This reduction is discussed in \cite{xieQuantitativeRelativeIndex2021}, in the second paragraph of the proof of Theorem 4.3. 
Since $Y_\pitchfork$ and $V_\pitchfork$ are cobordant, their equivariant coarse index coincides. 
\end{rmk}

Let $M$ be a closed spin manifold. Assume that $M$ has infinite $\mathcal{KO}_m$-multiwidth. For $n \in \bN$, pick a Riemannian $\square^m$-domain $V_n$ in the class $\mathcal{KO}_m$ which has the multiwidth $>n$ and is isometrically immersed to $M$. 
We use the letters $\pi_n$, $\Gamma_n$, $\mathbf{V}$, and $\widetilde{\mathbf{V}}$ in the same way as \cref{section:main}. 
Moreover, we define the maximal Roe algebras $C^*(\widetilde{\mathbf{V}})^{\boldsymbol{\pi}}$ and $C^*(\partial  \widetilde{\mathbf{V}} \subset \widetilde{\mathbf{V}})^{\boldsymbol{\pi}}$ in the same way. 
Note that, the same proof as \cref{lem:systole_propagation} shows that, the image of $\widetilde{V}_{n,R}$ in $\widetilde{M}$ has systole not less than $2C_R$. 
This means that the same construction as \cref{lem:lift} and \cref{lem:liftD} works, and hence we get the $\ast$-homomorphisms
\begin{align*}
    s &\colon C^*(\widetilde{M})^\Gamma \to 
C^*(\widetilde{\mathbf{V}})^{\boldsymbol{\pi}}/ C^*(\partial \widetilde{\mathbf{V}} \subset \widetilde{\mathbf{V}})^{\boldsymbol{\pi}}, \\
s &\colon D^*(\widetilde{M})^\Gamma \to 
D^*(\widetilde{\mathbf{V}})^{\boldsymbol{\pi}}/ C^*(\partial \widetilde{\mathbf{V}} \subset \widetilde{\mathbf{V}})^{\boldsymbol{\pi}},\\
s &\colon Q^*(\widetilde{M})^\Gamma \to 
Q^*(\widetilde{\mathbf{V}})^{\boldsymbol{\pi}}/ C^*(\partial \widetilde{\mathbf{V}} \subset \widetilde{\mathbf{V}})^{\boldsymbol{\pi}}.
\end{align*} 

Let $W_n:=\partial_{1,+}V$ and let $\widetilde{W}_n$ denote the $\pi_n$-Galois covering of $W_n$ (note that it is not necessarily the universal covering of $W_n$). 
Set $\mathbf{W}=\bigsqcup_{n \in \bN} W_n $ and $\widetilde{\mathbf{W}}=\bigsqcup_{n \in \bN} \widetilde{W}_n $. 
We remark that $W_n$ is a Riemannian $\square^{m-1}$-domain. 

Let $\slashed{D}_{\widetilde{\mathbf{V}}}$ and $\slashed{D}_{\widetilde{\mathbf{W}}}$ denote the Dirac operator on $\widetilde{\mathbf{V}}$ and $\widetilde{\mathbf{W}}$ respectively. 
In the same way as the previous section, these operators determine the $\KO$-classes $[\slashed{D}_{\widetilde{\mathbf{V}}}] \in \KO_{d+1}(Q^*(\widetilde{\mathbf{V}})^{\boldsymbol{\pi}}/Q^*(\partial \widetilde{\mathbf{V}} \subset \widetilde{\mathbf{V}})^{\boldsymbol{\pi}})$ and 
$[\slashed{D}_{\widetilde{\mathbf{W}}}] \in \KO_{d}(Q^*(\widetilde{\mathbf{W}})^{\boldsymbol{\pi}}/Q^*(\partial \widetilde{\mathbf{W}} \subset \widetilde{\mathbf{W}})^{\boldsymbol{\pi}})$ respectively.
\begin{lem}\label{lem:boundary}
There is a homomorphism
\begin{align*}
\partial \colon \KO_d \bigg( \frac{C^*(\widetilde{\mathbf{V}})^{\boldsymbol{\pi}}}{C^*(\partial \widetilde{\mathbf{V}} \subset \widetilde{\mathbf{V}})^{\boldsymbol{\pi}}} \bigg) \to \KO_{d-1} \bigg( \frac{C^*(\widetilde{\mathbf{W}})^{\boldsymbol{\pi}}}{C^*(\partial \widetilde{\mathbf{W}} \subset \widetilde{\mathbf{W}})^{\boldsymbol{\pi}}} \bigg)
\end{align*}
sending $\Ind_{\boldsymbol{\pi}}([\mathbf{\slashed{D}}_{\widetilde{\mathbf{V}}}])$ to $\Ind_{\boldsymbol{\pi}}([\mathbf{\slashed{D}}_{\widetilde{\mathbf{W}}}])$.
\end{lem}
\begin{proof}
Let $\widetilde{\mathbf{Z}}$ denote the closure of $\partial \widetilde{\mathbf{V}} \setminus \widetilde{\mathbf{W}}$. Note that $\widetilde{\mathbf{Z}} \cap \widetilde{\mathbf{W}} = \partial \widetilde{\mathbf{W}}$, and hence 
\[ C^*(\partial \widetilde{\mathbf{V}})^{\boldsymbol{\pi}}/C^*(\widetilde{ \mathbf{Z}} \subset \partial \widetilde{\mathbf{V}} )^{\boldsymbol{\pi}} \cong C^*(\widetilde{\mathbf{W}})^{\boldsymbol{\pi}}/C^*(\partial \widetilde{\mathbf{W}} \subset \widetilde{\mathbf{W}})^{\boldsymbol{\pi}}. \]
The Real C*-algebra extension
\[0 \to \frac{C^*(\partial \widetilde{\mathbf{V}})^{\boldsymbol{\pi}}}{C^*(\widetilde{ \mathbf{Z}} \subset \partial \widetilde{\mathbf{V}} )^{\boldsymbol{\pi}}} \to \frac{C^*(\widetilde{\mathbf{V}})^{\boldsymbol{\pi}}}{C^*(\widetilde{\mathbf{Z}}  \subset \widetilde{\mathbf{V}})^{\boldsymbol{\pi}}} \to \frac{C^*(\widetilde{\mathbf{V}})^{\boldsymbol{\pi}}}{C^*(\partial \widetilde{\mathbf{V}} \subset \widetilde{\mathbf{V}})^{\boldsymbol{\pi}}} \to 0 \]
induces the boundary map
\[\partial \colon \KO_d \Big(\frac{C^*(\widetilde{\mathbf{V}})^{\boldsymbol{\pi}}}{C^*(\partial \widetilde{\mathbf{V}} \subset \widetilde{\mathbf{V}})^{\boldsymbol{\pi}}} \Big) \to \KO_{d-1} \Big(\frac{C^*(\partial \widetilde{\mathbf{V}})}{C^*(\partial \widetilde{ \mathbf{Z}} \subset \partial \widetilde{\mathbf{V}} )}  \Big) \cong \KO_{d-1} \Big( \frac{C^*(\widetilde{\mathbf{W}})^{\boldsymbol{\pi}}}{C^*(\partial \widetilde{\mathbf{W}} \subset \widetilde{\mathbf{W}})^{\boldsymbol{\pi}}} \Big) \]
as desired. Moreover, the boundary of the same kind are also defined for $D^*$ and $Q^*$ coarse C*-algebras in the same way. Now, the commutativity of the diagram
\[
\xymatrix{
\KO_{d+1}\Big( \frac{Q^*(\widetilde{\mathbf{V}})^{\boldsymbol{\pi}}}{Q^*(\partial \widetilde{\mathbf{V}} \subset \widetilde{\mathbf{V}})^{\boldsymbol{\pi}}} \Big) \ar[r]^\partial  \ar[d]^{\Ind _{\boldsymbol{\pi}}} & 
\KO_d \Big( \frac{Q^*(\widetilde{\mathbf{V}})^{\boldsymbol{\pi}}}{Q^*(\partial \widetilde{\mathbf{V}} \subset \widetilde{\mathbf{V}})^{\boldsymbol{\pi}}} \Big) \ar[d]^{\Ind _{\boldsymbol{\pi}}}
\\
\KO_d \Big( \frac{C^*(\widetilde{\mathbf{V}})^{\boldsymbol{\pi}}}{C^*(\partial \widetilde{\mathbf{V}} \subset \widetilde{\mathbf{V}})^{\boldsymbol{\pi}}} \Big) \ar[r]^\partial & 
\KO_{d-1}\Big( \frac{C^*(\widetilde{\mathbf{V}})^{\boldsymbol{\pi}}}{C^*(\partial \widetilde{\mathbf{V}} \subset \widetilde{\mathbf{V}})^{\boldsymbol{\pi}}} \Big)
}
\]
and the boundary of Dirac is Dirac principle, $\partial [\slashed{D}_{\mathbf{V}}]=[\slashed{D}_{\mathbf{W}}]$, shows that $\partial (\Ind_{\boldsymbol{\pi}}([\mathbf{\slashed{D}}_{\widetilde{\mathbf{V}}}])) = \Ind_{\boldsymbol{\pi}}([\mathbf{\slashed{D}}_{\widetilde{\mathbf{W}}}])$ as desired. 
\end{proof}

\begin{proof}[Proof of \cref{thm:main2}]
For $1 \leq k\leq m$, let $W_n^k$ denote the face $\partial_{1,+}V \cap \cdots \partial _{k,+}V$, which is a $\square^{m-k}$-domain. 
Let us consider the $m-1$ iterated composition of the map defined in \cref{lem:boundary} as
\[
\mathclap{
\KO_d \Big(\frac{C^*(\widetilde{\mathbf{V}})^{\boldsymbol{\pi}}}{C^*(\partial \widetilde{\mathbf{V}} \subset \widetilde{\mathbf{V}})^{\boldsymbol{\pi}}} \Big) \xrightarrow{\partial} 
\KO_{d-1} \Big(\frac{C^*(\widetilde{\mathbf{W}^1})^{\boldsymbol{\pi}}}{C^*(\partial \widetilde{\mathbf{W}}^1 \subset \widetilde{\mathbf{W}}^1)^{\boldsymbol{\pi}}} \Big) \xrightarrow{\partial} 
\cdots \xrightarrow{\partial}
\KO_{d-m+1} \Big(\frac{C^*(\widetilde{\mathbf{W}}^{m-1})^{\boldsymbol{\pi}}}{C^*(\partial \widetilde{\mathbf{W}}^{m-1} \subset \widetilde{\mathbf{W}}^{m-1})^{\boldsymbol{\pi}}} \Big)_{\textstyle .}}
\]
We further compose the boundary map \eqref{eq:boundary_flat} and $\phi_+$. Finally we get
\[\phi_+ \circ \partial ^{m} \colon \KO_d \Big(\frac{C^*(\widetilde{\mathbf{V}})^{\boldsymbol{\pi}}}{C^*(\partial \widetilde{\mathbf{V}} \subset \widetilde{\mathbf{V}})^{\boldsymbol{\pi}}} \Big) \to \frac{\prod \KO_{d-m} (C^*(\widetilde{V}_{n,\pitchfork})^{\pi_n})}{ \bigoplus \KO_{d-m}(C^*(\widetilde{V}_{n,\pitchfork})^{\pi_n})} _{\textstyle .}\]
By \cref{lem:boundary} and the proof of \cref{thm:main}, this map sends the  equivariant coarse index $\Ind_{\boldsymbol{\pi}}(\slashed{D}_{\widetilde{\mathbf{V}}})$ to $(\Ind _{\pi_n}(\slashed{D}_{\widetilde{V}_{n,\pitchfork}}))_{n}^\flat$, which does not vanish by assumption. This finishes the proof. 
\end{proof}

\bibliographystyle{alpha}
\bibliography{ref.bib}

\end{document}